\documentclass[11pt]{amsart}
\usepackage{geometry}                
\usepackage{graphicx}
\usepackage{amsrefs}
\usepackage{amssymb}
\usepackage{mathtools}
\usepackage[matrix,arrow,ps]{xy}
\usepackage{hyperref}
\usepackage{epstopdf}
\usepackage{faktor}

\newcommand{\C}{\mathbb{C}}
\newcommand{\CN}{\mathbb{C}^N}
\newcommand{\CNp}{\mathbb{C}^{N^\prime}}

\newcommand{\R}{\mathbb{R}}

\newcommand{\N}{\mathbb{N}}

\newcommand{\todo}[1]{}

\newlength{\extendaxesby}

\DeclareMathOperator{\id}{id}

\DeclareMathOperator{\real}{Re}

 \newtheorem{thm}{Theorem}
\newtheorem{theorem}[thm]{Theorem}

\newtheorem{lem}[thm]{Lemma}
\newtheorem{lemma}[thm]{Lemma}

\newtheorem{cor}[thm]{Corollary}

\theoremstyle{definition}

\newtheorem{definition}[thm]{Definition}

\newtheorem{remark}[thm]{Remark}

\begin{document}

\title[Infinitesimal and local rigidity of mappings]{Infinitesimal and local rigidity \\
 of mappings of CR manifolds}
\author{Giuseppe della Sala}
\address{Department of Mathematics, American University of Beirut (AUB)}
\email{gd16@aub.edu.lb}
\author{Bernhard Lamel}
\address{Fakult\"at f\"ur Mathematik, Universit\"at Wien}
\email{bernhard.lamel@univie.ac.at}
\author{Michael Reiter}
\address{Fakult\"at f\"ur Mathematik, Universit\"at Wien}
\email{m.reiter@univie.ac.at}

\thanks{The first author was supported by the FWF project P24878-N25, and would also like to
thank the Center for Advanced Mathematical Sciences (CAMS) at AUB. The second author was supported by the FWF-Project I382 and QNRF-Project NPRP 7-
511-1-098. The third author was supported by the FWF-Project P28873}

\begin{abstract}
A holomorphic mapping $H$ between two real-analytic CR manifolds $M$ and $M'$ is said to be locally rigid if any other holomorphic map $F\colon M \to M'$ which is close enough to $H$ is obtained by composing $H$ with suitable automorphisms of $M$ and $M'$. With the aim of reducing the local rigidity problem to a linear one,  we provide  sufficient infinitesimal conditions. Furthermore we study some topological properties of the action of the automorphism group on the space of nondegenerate mappings from $M$ to $M'$. 
\end{abstract}

\maketitle

\section{Introduction}
\label{intro}

Let $M \subset \C^N$ be a (real-analytic) generic submanifold. We define ${\rm Aut}_0 (M)$ to be the group of the germs $\sigma$ of biholomorphic maps $\mathbb C^N\to \mathbb C^N$, defined around $0$, such that $\sigma(0) = 0$ and $\sigma(M)\subset M$. We denote the Lie algebra of ${\rm Aut}_0 (M)$ by $\mathfrak{hol}_0(M)$.

Let now $M$ and $M'$ be germs of generic real-analytic CR submanifolds in $\C^N$ and $\C^{N'}$ respectively, and let $\mathcal H(M,M')$ denote the space of germs of holomorphic mappings which send $M$ into $M'$.  The group  $G=  {\rm Aut}_0 (M) \times {\rm Aut}_0 (M') $, which we call the \emph{isotropy group}, acts on $\mathcal H(M,M')$ via $H \mapsto \sigma' \circ H \circ \sigma^{-1}$, where $(\sigma, \sigma')\in G$ and $H \in \mathcal H(M,M')$. If we endow all of these sets with their natural (inductive limit) topologies, they become topological spaces and groups, respectively. We are interested in studying the topological properties of this (continuous) group action. More precisely we would like to continue our study of local rigidity of mappings, a notion we introduced in \cite{dSLR15}: we say a map $H\in \mathcal H(M,M')$ is locally rigid if it projects to an isolated point in the quotient $\mathcal H(M,M') / G$ (for a formal definition, see Definition \ref{locrig}).

Our aim is to provide linear -- and thus easier to compute --
sufficient conditions for local rigidity. In order to state a criterion in this direction, 
we say a holomorphic section $V$ of $T^{1,0}(\mathbb C^{N'})|_{H(\mathbb C^N)}$, 
which vanishes at $0$, is an \emph{infinitesimal deformation} of $H \in \mathcal H(M,M')$ 
if $\real V$ is tangent to $M'$ along $H(M)$ (for the formal definition see Definition \ref{def:infdef}). We denote the set of infinitesimal deformations of $H$ by $\mathfrak {hol}_0 (H)$; it forms a real vector space.

We are particularly interested in sets of mappings satisfying certain generic nondegeneracy conditions introduced in \cite{La}: in Definition \ref{defNondeg} below we formally introduce the set of finitely nondegenerate mappings.

We are now going to state our main results, which are built on the approach and techniques of our recent work \cite{dSLR15}. In this paper we exploit our method of infinitesimal deformations to study more general situations. The first result is a generalization of Theorem~1 of \cite{dSLR15}.

\begin{theorem}\label{infTrivial}
Let $M$ be a germ of a generic minimal real-analytic submanifold through $0$ in $\C^N$, and $M'$ be 
 a germ of a generic real-analytic submanifold in $\C^{N'}$. Let $H\in \mathcal H(M,M') $ be a germ 
 of a finitely nondegenerate map satisfying
  $\dim_{\mathbb R} \mathfrak {hol}_0 (H) = 0$. Then $H$ is an isolated point in $\mathcal H(M,M')$, and  in particular, $H$ is locally rigid. 
 \end{theorem}

In the next result we are going to relax the assumption $\dim_{\mathbb R} \mathfrak {hol}_0 (H) = 0$. In order to do so we will restrict to the case when $M \subset \C^N$ and $M' \subset \C^{N'}$ are strictly pseudoconvex hypersurfaces, so that $M$ and $M'$ have CR dimension $n=N-1$ and $n'=N'-1$ respectively. Moreover, since the embeddings of spheres have been studied extensively \cites{We, Fa2, Da, Hu, HJ, Ji, Le, Re2, Re3} (see \cite{dSLR15} for a more thorough discussion and additional references), we assume that $M$ or $M'$ is not biholomorphically equivalent to a sphere.

The following result is a generalization of Theorem 2 of \cite{dSLR15} in the setting of strictly pseudoconvex hypersurfaces.

\begin{theorem}\label{suffcon2intro}
Let $M \subset\C^N$ and $M' \subset \C^{N'}$ be germs of strictly pseudoconvex real-analytic hypersurfaces through $0$, where at least one of $M$ or $M'$ is not spherical. If $H\in \mathcal H(M,M') $ is a germ of a $2$-nondegenerate map that satisfies
  $\dim_{\mathbb R} \mathfrak {hol}_0 (H) = \dim_{\mathbb R}\mathfrak{hol}_0(M')$, then $H$ is locally rigid. 
\end{theorem}

We note that the assumption of $2$-nondegeneracy implies that $N'\leq \frac{N(N+1)} 2$.

The outline of the paper is as follows: The proofs of our main results are provided in the very last section \ref{proof}. Before, we fix notation in  section \ref{prelim} and give a jet parametrization result for finitely nondegenerate maps in section \ref{s:jetparam}. Then, we study infinitesimal deformations in section \ref{s:infdef} and deduce some crucial properties of the action of isotropies on the space of maps in section \ref{s:propiso}.

\section{Preliminaries}
\label{prelim}

This section is devoted to introduce some standard notation. For details and proofs, we refer the reader to e.g. \cite{BER2}.

For a generic real-analytic CR submanifold $M\subset \mathbb C^N$ we denote by $n$ its CR dimension and by $d$ its real codimension so that $N=n+d$. It is well-known (cf. \cite{BER2}) that one can choose \emph{normal coordinates}
 $(z,w)\in \mathbb C^n_z\times \mathbb C^d_w=\mathbb C^N$ such that the {\em complexification} $\mathcal M\subset \mathbb C^{2N}_{z,\chi,w,\tau}$ of $M$ is 
 given by
\[ w = Q(z,\chi, \tau), \ \ {\rm (equivalently:} \ \tau = \overline Q(\chi,z, w) {\rm )},\]
for a suitable germ of holomorphic map $Q:\mathbb C^{2n+d}\to\mathbb C^{d}$ satisfying the following equations
 \begin{equation}
 \label{NormalCoord}
Q(z,0,\tau) \equiv Q(0,\chi,\tau) \equiv \tau, \ \ \ \ Q(z,\chi, \overline Q (\chi, z, w)) \equiv w.
   \end{equation}

Given a defining equation $\rho' \in (\C\{Z',\bar Z'\})^{d'}$ for $M'$, we recall that
 a germ of a mapping $H(z,w)\in (\mathbb C\{z,w\})^{N'}$ with $H(0,0)=0$, belongs to $\mathcal H(M,M')$ if and only if it solves the \emph{mapping equation}
\begin{align}
\label{mapeq}
\rho'(H(z,w),\overline H(\chi,\tau))= 0 \quad {\rm for}\  w = Q(z,\chi,\tau).
\end{align}

We endow $(\C\{Z\})^{N'}$, which we consider as the space of germs at $0$ of holomorphic maps from $\C^N$ to $\C^{N'}$,  with the natural direct limit topology and consider $\mathcal H(M,M') \subset (\C\{Z\})^{N'}$ with the induced topology.

We denote by $L_j$ and $\bar L_j$, $j=1,\ldots, n$, a commuting basis of the germs of CR and anti-CR vector fields, respectively, tangent to $\mathcal M$. Furthermore it will be convenient to consider the following vector fields, which are also tangent to $\mathcal M$:
\[T_\ell = \frac{\partial}{\partial w_\ell} + \sum_{k=1}^d \overline Q^k_{w_\ell}(\chi,z,w) \frac{\partial}{\partial \tau_k}, \ \ \ \ S_j = \frac{\partial}{\partial z_j} + \sum_{k=1}^d\overline Q^k_{z_j}(\chi,z,w) \frac{\partial}{\partial \tau_k} \]
where $1\leq \ell \leq d$, $1\leq j\leq n$. 

In the sequel we denote by $\N = \{1,2,3,\ldots\}$ the set of natural numbers and write $\N_0 = \N \cup \{0\}$. The notion of nondegeneracy we are interested in was introduced in \cite{La}. For our purposes we will also need a slightly weaker one:

\begin{definition}\label{defNondeg}
Let $M'=\{\rho' = 0\}$, where $\rho' =(\rho_1',\ldots, \rho'_{d'}) \in  (\mathbb C\{Z',\zeta'\})^{d'}$ is a local defining function for $M'$. Given a holomorphic map
 $H=(H_1,\ldots,H_{N'})\in (\mathbb C\{z,w\})^{N'}$, a fixed sequence $(\iota_1,\ldots,\iota_{N'})$ of 
 multiindices $\iota_m\in\N_0^n$ and integers $\ell^1,\ldots, \ell^{N'}$ with $1 \leq \ell^j \leq d'$, we consider the determinant
 \begin{equation} \label{folcon}
s = \det \left(\begin{array}{ccc} 
%r_1^1(H(z,w),\overline H(\chi,\tau)) & \cdots & r^1_{N'}(H(z,w),\overline H(\chi,\tau)) \\
 L^{\iota_1}\rho'_{\ell^1, Z_1'}(H(z,w),\overline H(\chi,\tau)) & \cdots & L^{\iota_1}\rho'_{\ell^1,Z_{N'}'}(H(z,w),\overline H(\chi,\tau)) \\  \vdots & \ddots & \vdots \\
 L^{\iota_{N'}}\rho'_{\ell^{N'}, Z_1'}(H(z,w),\overline H(\chi,\tau)) & \cdots & L^{\iota_{N'}} \rho'_{\ell^{N'}, Z_{N'}'}(H(z,w),\overline H(\chi,\tau))\end{array}\right).
\end{equation}

We define the open set $\mathcal F_{k} \subset \mathcal H(M,M')$ as the set of maps $H$  for which there exists a sequence of multiindices $(\iota_1,\ldots,\iota_{N'})$ with $k = \max_{1 \leq m \leq N'}|\iota_m|$ and integers $\ell^1,\ldots, \ell^{N'}$ as above  such that $s(0)\neq 0$.
We will say that $H$ with $H(M) \subset M'$ is {\em $k_0$-nondegenerate} if $k_0 = \min \{ k\colon H \in \mathcal{F}_k \}$ is a finite number.
\end{definition}

 Note that the definition of both $\mathcal F_{k_0}$ and the space of $k_0$-nondegenerate maps are independent of the choice of coordinates (see 
\cite[Lemma 14]{La}), hence these spaces are invariant under the action of $G$. Also notice that in the setup of Theorem \ref{suffcon2intro} the space of $2$-nondegenerate maps coincides with the set $\mathcal F_2$.

Our first main goal in this paper is to study the following property:

\begin{definition}\label{locrig}
Let $M$ and $M'$ be germs of submanifolds in $\mathbb C^N$ (resp. $\mathbb C^{N'}$) around $0$, 
and let $H$ be a mapping of $M$ into $M'$. 
We say that $H$ is \emph{locally rigid} if $H$ projects to an isolated point in the 
quotient $\faktor{\mathcal H(M,M')}{G}$ of the space $\mathcal H(M,M')$ of holomorphic mappings from $M$ to $M'$ with respect to the group of isotropies $G$.
\end{definition}

\begin{remark}\label{rem:equcon}
 It is easy to show that $H \in \mathcal H(M,M')$ is locally rigid according to the definition above if and only if  there exists a neighborhood $U$ of $H$ in $(\mathbb C\{Z\})^{N'}$ such that for every $\hat H\in \mathcal H(M,M')\cap U$ there is $g\in G$ such that $\hat H=g H$. In other words, $H$ is locally rigid if and 
 only if all the maps in $\mathcal H(M,M')$ which are close enough to $H$ are equivalent to $H$ (see Remark 12 in \cite{dSLR15}).
\end{remark}

\section{Jet parametrization}
\label{s:jetparam}

In order to prove our main theorems  we will show that in an appropriate sense, the infinitesimal deformations can be considered as a tangent space, by deducing a jet parametrization result for maps in $\mathcal F_{k}$ based on the work in \cites{La, BER99, JL}. First we will introduce some notation.

Let $H: \C^N \to \C^{N'}$ be a germ of a holomorphic map and let $k$ be an integer. We denote by $j_0^k H$ the \emph{$k$-jet of $H$ at $0$}, that is the collection of all derivatives of order $\leq k$ of the components of $H$ at $0$. The space of all $k$-jets at $0$ will be denoted by $J_0^k$ (we drop the dependence on $N$ and $N'$, which will remain fixed, for better readability). We denote by $\Lambda$ coordinates in $J_0^k$ and write
$\Lambda = (\Lambda',\Lambda_{N'}) = (\Lambda_1,\ldots,\Lambda_{N'-1},\Lambda_{N'})$ with $\Lambda_j=(\Lambda_j^{\alpha,\beta})$, where $\alpha \in \N_0^{n}, \beta \in \N_0^d$ and $0\leq |\alpha|+|\beta| \leq k$. We have  
\[ \Lambda = j_0^k H \text{ if  and only if } \Lambda_j^{\alpha,\beta} = \frac{1}{\alpha!\beta!} 
\frac{\partial^{|\alpha|+|\beta|} H_j}{\partial z^\alpha \partial w^\beta} (0).  \]
We can identify $k$-th order jets with polynomial maps of degree at most $k$ (taking $\CN$ into $\CNp$) and will do so freely in the sequel. In particular, the composition of a jet with a jet is defined, as well as the composition of jets with other maps, provided that the source and target dimensions allow it. 
 
We also need to recall the definition of certain subsets of $\mathbb C^N$, commonly referred to as the \emph{Segre sets}. In order to do so we need to introduce some notation. For any $j\in \N$ let $(x_1,\ldots, x_j)$ ($x_\ell\in \mathbb C^n$) be coordinates for $\mathbb C^{nj}$. The \emph{Segre map} of order $q\in \N$ is the map $S^q_0:\mathbb C^{nq}\to \mathbb C^N$ inductively defined as follows:
 \[S^1_0(x_1) = (x_1,0), \ \ S^q_0(x_1,\ldots,x_q) = \left (x_1, Q\left(x_1,\overline S^{q-1}_0(x_2,\ldots,x_q) \right)\right)\]
where we denote by $\overline S^{q-1}_0$ the power series whose coefficients are conjugate to the ones of $S^{q-1}_0$ and $Q$ is a map as given in \eqref{NormalCoord}. In particular if $w-Q(z,\chi,\tau)=0$ is a local defining equation of the complexification of a CR submanifold $M \subset \C^N$, we say the Segre map $S^q_0$ is associated to $M$. The $q$-th Segre set $\mathcal S^q_0\subset \mathbb C^N$ is then the image of the map $S^q_0$. In what follows we will use the notation $x^{[j;k]} = (x_j,\ldots,x_k)$. It is known from the Baouendi-Ebenfelt-Rothschild minimality criterion \cite{BER1} that if $M$ is minimal at $0$, then
$S_0^j$ is generically of full rank if $j$ is large enough. We recall that a germ of a CR submanifold $M \subset \C^N$ is called \emph{minimal at $p \in M$} if there is no germ of a CR submanifold $\tilde M \subsetneq M$ of $\C^N$ through $p$ having the same CR dimension as $M$ at $p$.

\begin{theorem}\label{jetparam} 
Let $M\subset \mathbb C^N$ be the germ of a real-analytic, generic minimal submanifold, $0\in M$, and let $M'\subset \mathbb C^{N'}$ be a real-analytic generic submanifold germ. Let $k_0 \in \N$ and $\mathbf t \leq d+1$ be the minimum integer, such that the Segre map $S^{\mathbf t}_0$ of order $\mathbf t$ associated to $M$  is generically of full rank. First suppose that $\mathbf t$ is even. 
There exists a finite collection of polynomials
$q_{j}(\Lambda)$ on $J_0^{\mathbf t k_0}$ for $j \in J$, where $J$ is a suitable finite index set, open neighborhoods $\mathcal U_j$ of  $\{0\} \times U_j$ in $\mathbb C^N \times J_0^{\mathbf t k_0}$, where $U_j = \{q_j\neq 0\}$ and holomorphic maps
$\Phi_j \colon \mathcal U_j \to \C^{N'} $ satisfying $\Phi_j(0,\Lambda) = 0$, which are of the form 
\begin{align} \label{rationalJetParam}
\Phi_j (Z,\Lambda) =  \sum_{\alpha\in \N_0^N} \frac{p_j^{\alpha} (\Lambda)}{q_j(\Lambda)^{d^j_\alpha}} Z^\alpha, \quad  p_j^\alpha, q_j \in \C[\Lambda], \quad d^j_\alpha\in\N_0,
\end{align}
such that the following holds:
\begin{itemize}
\item For every $H\in\mathcal F_{k_0}$, in particular for every $k_0$-nondegenerate map $H$, there exists $j \in J$ such that $j_0^{\mathbf t k_0} H \in U_j$.
\item For every $H\in\mathcal F_{k_0} $ with $j_0^{\mathbf t k_0} H \in U_j$ we have \[ H(Z) = \Phi_j (Z,j_0^{\mathbf t  k_0} H).\]
\end{itemize}

In particular, there exist (real) polynomials $c^j_k$, $k\in\N$ on $J_0^{\mathbf t  k_0}$ such that
\begin{align}
\label{defEquationJetParam}
 j_0^{\mathbf t k_0} \mathcal F_{k_0} = \bigcup_{j \in J}\{ \Lambda\in J_0^{\mathbf t k_0} \colon q_j(\Lambda) \neq 0, \, c^j_k (\Lambda, \bar \Lambda) =0 \}.
 \end{align}

 Analogous statements hold for $\mathbf t$ odd, where in this case all $p_j^{\alpha}$ and $q_j$ in the expansion of $\Phi_j$ in \eqref{rationalJetParam} depend antiholomorphic on $\Lambda$.
\end{theorem}

The proof of Theorem \ref{jetparam} will be split up into several lemmas. We define $K(t) = |\{ \alpha \in \N_0^N: |\alpha| \leq t\}|$.

\begin{lemma}\label{lem:basicIdentity}
Let $M$ and $M'$ be as before. Fix multiindices $(\iota_1,\ldots,\iota_{N'})$ and integers $\ell^1,\ldots, \ell^{N'}$ as above. Let $k_0 = \max_{1\leq m \leq N'} |\iota_m|$. There is a holomorphic map $\Psi: \C^N \times \C^N \times \C^{K(k_0) N'} \to \C^{N'}$ such that for every holomorphic map $H: \C^N \to \C^{N'}$ satisfying \eqref{mapeq} and $s(0)\neq 0$, where $s$ is given as in \eqref{folcon}, we have
\begin{align}
\label{basicIdentity}
H(Z) = \Psi(Z,\zeta, \partial^{k_0} \bar H(\zeta)), 
\end{align}
for $(Z,\zeta)$ in a neighborhood of $0$ in $\mathcal M$, where $\partial^{k_0}$ denotes the collection of all derivatives up to order $k_0$. Furthermore there exist polynomials $P_{\alpha, \beta}, Q$ and integers $e_{\alpha,\beta}$ such that 
\begin{align}
\label{basicIdentityRational}
\Psi(Z,\zeta,W) = \sum_{\alpha,\beta \in \N_0^N} \frac{P_{\alpha,\beta}(W)}{Q^{e_{\alpha,\beta}}(W)} Z^{\alpha} \zeta^{\beta}.
\end{align}
\end{lemma}

In Lemma \ref{lem:basicIdentity} the statement up until \eqref{basicIdentity} is a reformulation of Prop. 25 in \cite{La}. The expansion in \eqref{basicIdentityRational} follows from the way the implicit function theorem is applied in the proof of Prop. 25, in a similar fashion as in \cites{BER99, JL}.

From now on all the jet parametrization mappings that will appear in the following lemmas will depend on the multiindices and integers fixed in Lemma \ref{lem:basicIdentity}. For the sake of readability we will omit to write this dependence explicitly.

\begin{lem}\label{lem:derivBasicIdentity}
Under the assumptions of Lemma \ref{lem:basicIdentity} the following holds: For all $\ell \in \N$ there exists a holomorphic mapping $\Psi_{\ell}: \C^N \times \C^N \times \C^{K(k_0 + \ell) N'} \to \C^{N'}$ such that for every holomorphic map $H: \C^N \to \C^{N'}$ satisfying \eqref{mapeq} and $s(0)\neq 0$, where $s$ is given as in \eqref{folcon}, we have
\begin{align}
\label{derivBasicIdentity}
\partial^{\ell} H(Z) = \Psi_{\ell}(Z,\zeta, \partial^{k_0+ \ell} \bar H(\zeta)), 
\end{align}
for $(Z,\zeta)$ in a neighborhood of $0$ in $\mathcal M$, where $\partial^{\ell}$ denotes the collection of all derivatives up to order $\ell$. Furthermore there exist polynomials $P^{\ell}_{\alpha, \beta},Q_{\ell}$ and integers $e^{\ell}_{\alpha,\beta}$ such that 
\begin{align}
\label{derivBasicIdentityRational}
\Psi_{\ell}(Z,\zeta,W) = \sum_{\alpha,\beta \in \N_0^N} \frac{P^{\ell}_{\alpha,\beta}(W)}{Q_{\ell}^{e^{\ell}_{\alpha,\beta}}(W)} Z^{\alpha} \zeta^{\beta}.
\end{align}

\end{lem}

Lemma \ref{lem:derivBasicIdentity} follows by differentiating \eqref{basicIdentity} along the vector fields $S$ and $T$ introduced above, see Cor. 26 of \cite{La}.

The next step is to evaluate (\ref{basicIdentity}) along the Segre sets.

\begin{cor}\label{cor:iterationSegre}
Under the assumptions of Lemma \ref{lem:basicIdentity} the following holds: For fixed $q \in \N$ there exists a holomorphic mapping $\varphi_{q}: \C^{qn} \times \C^{K(q k_0) N'}\to \C^{N'}$ such that for every holomorphic map $H: \C^N \to \C^{N'}$ satisfying \eqref{mapeq} and $s(0)\neq 0$, where $s$ is given as in \eqref{folcon}, we have
\begin{align}
\label{iterationSegre}
H(S^q_0(x^{[1;q]})) = \varphi_{q}(x^{[1;q]}, j_0^{q} H).
\end{align}
Furthermore there exist (holomorphic) polynomials $R^{q}_{\gamma},S_{q}$ and integers $m^q_{\gamma}$ such that 
\begin{equation}
\label{iterationSegreRational}
\varphi_{q}(x^{[1;q]},\Lambda) =\begin{cases}  \sum_{\gamma \in \N_0^{qn}} \frac{R^{q}_{\gamma}( \Lambda)}{S_{q}^{m^{q}_{\gamma}}(\Lambda)} (x^{[1;q]})^{\gamma}  & q \text{ even }\\
\sum_{\gamma \in \N_0^{qn}} \frac{R^{q}_{\gamma}( \bar \Lambda)}{S_{q}^{m^{q}_{\gamma}}(\bar \Lambda)} (x^{[1;q]})^{\gamma} & q \text{ odd }.\end{cases}
\end{equation}

\end{cor}

\begin{proof}
Fix $q\in \N$. We begin by putting $Z=S^q_0(x^{[1;q]})$ and $\zeta = \bar S^{q-1}_0(x^{[2;q]})$, so that $(Z,\zeta)\in \mathcal M$, in the identity (\ref{basicIdentity}), in order to obtain
\begin{align}\label{firsteval}
H(S^q_0(x^{[1;q]})) = \Psi(S^q_0(x^{[1;q]}),\bar S^{q-1}_0(x^{[2;q]}), \partial^{k_0} \bar H(\bar S^{q-1}_0(x^{[2;q]}))).
\end{align}
This equation means that one can determine the value of any solution of (\ref{mapeq}), at least along $\mathcal S^q_0$, by knowing the values of its derivatives along $\mathcal S^{q-1}_0$. To determine the latter, we put $\zeta = \overline S^{q-1}_0(x^{[2;q]})$ and $Z = S^{q-2}_0(x^{[3;q]})$ (again so that $(Z,\zeta) \in \mathcal M$) in the conjugate of (\ref{derivBasicIdentity}):
\begin{equation}\label{secondeval}
\partial^{\ell} \bar H(\overline S^{q-1}_0(x^{[2;q]})) = \bar \Psi_{\ell}(\overline S^{q-1}_0(x^{[2;q]}),S^{q-2}_0(x^{[3;q]}), \partial^{k_0+ \ell} H(S^{q-2}_0(x^{[3;q]}))).
\end{equation} 
%Note that in this case, $\overline s (0)=\overline {s(0)}$ is just a constant.
By substituting (\ref{secondeval}) for $\ell=k_0$ into (\ref{firsteval}), we get that the values of $H$ along $\mathcal S^q_0$ are determined by the values of their $2k_0$-th order jet along $\mathcal S^{q-2}_0$. Iterating this argument $q$ times we prove \eqref{iterationSegre}. To show \eqref{iterationSegreRational} we use \eqref{basicIdentityRational} and \eqref{derivBasicIdentityRational} at every step; the desired expansion follows by a (cumbersome but) straightforward computation. In particular one derives that $j_0^{q k_0} H \in \{S_q \neq 0\}$, whenever $s(0) \neq 0$. For more details see \cites{BER99, JL}
\end{proof}

\begin{proof}[Proof of Theorem \ref{jetparam}] By the choice of $\mathbf t\leq d+1$, the Segre map $S^{\mathbf t}_0$ is generically of maximal rank, and we can therefore define the finite number $\nu (S^{\mathbf t}_0)$ as the minimum order of vanishing of minor of maximal size of  the  Jacobian of $S^{\mathbf t}_0$.

We can thus appeal to Theorem~5 from \cite{JL} and obtain that there exist a neighborhood $\mathcal V$ of $S^{\mathbf t}_0$ in $(\mathbb C\{x^{[1;\mathbf t]}\})^N$ and a holomorphic map 
\[\phi:\mathcal V\times \mathbb C\{x^{[1;\mathbf t]}\} \to \mathbb C\{Z\}\]
such that $\phi(A,h\circ A) = h$ for all $A\in \mathcal V$ with $\nu(A) = \nu(S^{\mathbf t}_0)$, and for all $h\in\mathbb C\{Z\}$. %Furthermore, the map $\Phi$ is linear in the second factor.

Now, define $J$ as the set of all the sequences of multiindices $(\iota_1,\ldots,\iota_{N'})$ and integers $\ell^1,\ldots, \ell^{N'}$ as in Lemma \ref{lem:basicIdentity} with $k_0 = \max_{1\leq m \leq N'} |\iota_m|$. 
For any $j \in J$, Corollary \ref{cor:iterationSegre} with $q=\mathbf t$ provides the existence of a map $\varphi_{\mathbf t} = \varphi_{\mathbf t,j}$ satisfying \eqref{iterationSegre}. We set $\Phi_j(\cdot,\Lambda) = \phi(S^{\mathbf t}_0, \varphi_{\mathbf t,j}(\cdot,\Lambda))$ so that by the
properties of $\varphi_{\mathbf t,j}$ and $\phi$ the map $\Phi_j$ depends holomorphically on $\Lambda = j_0^{\mathbf t} H$ (or $\bar \Lambda$, respectively, if $j$ is odd). By setting $q_j(\Lambda,\bar \Lambda) = S_{\mathbf t,j}(\Lambda)$, where $S_{\mathbf t,j}$ is given in \eqref{iterationSegreRational}, a direct computation using \eqref{iterationSegreRational} and Thm. 5 in \cite{JL} (more precisely (42) of Thm. 6 of \cite{JL}), allows to derive the expansion in \eqref{rationalJetParam}.

It follows from Corollary \ref{cor:iterationSegre} and Thm.~5 in \cite{JL} that
 $H(Z) = \Phi_j(Z, j^{\mathbf tk_0}_0H)$
  whenever $H$ is a solution of (\ref{mapeq}) and $s(0)\neq 0$, where $s$ is given as in \eqref{folcon} with the sequence of multiindices corresponding to $j$. In particular if $H$ is $k_0$-nondegenerate by definition there exists $j \in J$ such that $s(0)\neq 0$, which by the arguments above is equivalent to the condition $j_0^{\mathbf t k_0} H \in U_j$.
  
Finally, the
remaining statement can be proved by setting
 $H(Z) =
\Phi_j(Z,\Lambda)$ in (\ref{mapeq}) and expanding it as a power series in
$(z,\chi,\tau)$: the coefficients of this power series depend polynomially on
 $\Lambda,\overline \Lambda$, so that the defining equations (\ref{defEquationJetParam}) can be obtained
by setting all the coefficients to $0$. 
\end{proof}

\section{Infinitesimal deformations}
\label{s:infdef}

In the following we refer to the notation of Theorem \ref{jetparam}. For any $j \in J$ let $A_j\subset U_j$ be the real-analytic set defined as
\begin{equation}\label{defas}
A_j = \{\Lambda\in  U_j \colon c^j_k(\Lambda,\overline \Lambda) = 0, k\in \N\}.
\end{equation}
By Theorem \ref{jetparam} putting $A \coloneqq \bigcup_j A_j$ we have $j_0^{\mathbf t k_0}(\mathcal F_{k_0}) = A$; in particular $A$ contains the set of $\mathbf t k_0$-jets of all $k_0$-nondegenerate mappings of $M$ into $M'$.
In fact we can say more:

\begin{lemma}\label{PhiHomo}
For every $j\in J$ we define $\mathcal F_{k_0, j} \coloneqq \mathcal F_{k_0} \cap (j_0^{\mathbf t k_0})^{-1}(U_j)$.
The map $\Phi_j:A_j\to\mathcal F_{k_0,j}$ is a homeomorphism. 
\end{lemma}

This result is proved exactly as Lemma 19 in \cite{dSLR15} as a direct consequence of Theorem \ref{jetparam}.

For each $j \in J$ the restriction of $\Phi_j$ to $\mathcal U_j \cap (\mathbb C^N \times A_j)$ gives rise to a map 
\[A_j \ni \Lambda \to \Phi_j(\Lambda) \in (\mathbb C\{Z\})^{N'}, \ \Phi_j(\Lambda) (Z) = \Phi_j(Z,\Lambda)\]
from $A_j$ to the space $(\mathbb C\{Z\})^{N'}$.

 Let $X\subset A_j$ be any regular (real-analytic) submanifold, and fix $\Lambda_0 \in X$. In what follows we focus on $\Phi_j|_X$. 
 %By Theorem \ref{jetparam}, the image of $X$ through $\Phi_j$ contains $(j_0^{\mathbf t k_0})^{-1}(X)\cap \mathcal F_{k_0}$. 
 Note that if we restrict to a small enough neighborhood of $\Lambda_0$ in $X$ (which we again denote by $X$)
 the maps $\Phi_j(\Lambda) $ for all $\Lambda\in X$ all have a common radius of convergence $R$, 
 so that we can consider the restriction of $\Phi_j$ to $X$ as a map 
 valued in  the Banach space ${\rm Hol}(\overline{B_R(0)}, \mathbb C^{N'})$, the space of holomorphic mappings from $B_R(0)$ to $\C^{N'}$, which are continuous up to $\overline{B_R(0)}$, where $B_R(0)$ denotes the ball with radius $R >0$ in $\C^N$.

We also remark that the map $\Phi_j: X \to (\mathbb C\{Z\})^{N'}$ is of class $C^\infty$. We consider its Fr\'echet derivative $D \Phi_j(\Lambda_0)$ at $\Lambda_0$ as a map $T_{\Lambda_0}X \to T_{\Phi_j(\Lambda_0)}(\mathbb C\{Z\})^{N'}\cong (\mathbb C\{Z\})^{N'}$.
%where we denote by $\frac{\partial \Phi}{\partial \Lambda ' }$ the differential of the map $\Phi(z,w,\Lambda)|_N$, evaluated at $\Lambda'$. 
 
%Write the components of  $\Phi(\Lambda_0)$ as $(\Phi_1, \Phi_2, \Phi_3)$.
We will need a special subspace of $T_{\Phi_j(\Lambda_0)}(\mathbb C\{Z\})^{N'}$. The following definition, already stated in Section \ref{intro}, was first given in \cite{CH}, see also \cite{dSLR15}.

\begin{definition}\label{def:infdef} Let $M$ and $M'$ be as above and 
$H\colon (\mathbb C^N,0) \to (\mathbb C^{N'},0)$ a map with $H(M) \subset M'$. Then a vector 
\[V =  \sum_{j=1}^{N'} \alpha_j(Z)\frac{\partial}{\partial Z_j'} \in T_{H}(\mathbb C\{Z\})^{N'}\] 
is an \emph{infinitesimal deformation of $H$} if the real part of $V$ is tangent to $M'$ along $H(M)$, i.e.\ if 
for one (and hence every) defining function $\rho'=(\rho'_1,\ldots, \rho'_{d'})$ of $M'$, the components of $V$ satisfy the following linear system 
\begin{equation}\label{infdef}
{\rm Re}\left(\sum_{j=1}^{N'}\alpha_j(Z)\frac{\partial \rho'_{\ell}}{\partial Z_j'}(H(Z),\overline {H(Z)})\right)=0 \ \ {\rm for} \ Z\in M, \ \ell = 1, \ldots, d'. 
\end{equation}
We denote this subspace of $T_{H}(\mathbb C\{Z\})^{N'}$ by $\mathfrak {hol}_0 (H)$.
\end{definition}

With the same proof as in \cite{dSLR15} we derive the following property, which motivates the definition above:
\begin{lemma}\label{contain}
The image of $T_{\Lambda_0}X$ by $D\Phi_j(\Lambda_0)$ is contained in $\mathfrak {hol}_0 (\Phi_j(\Lambda_0))$.
\end{lemma}

The next lemmas give some properties of infinitesimal deformations that will be needed in section \ref{proof} to give the proofs our main theorems.
The first lemma comes from the jet parametrization for solutions of  \eqref{infdef} obtained in section 5 in \cite{dSLR15}. Its proof is precisely the one of Cor. 32 in \cite{dSLR15} using Prop. 29 instead of Prop. 31 and keeping $Q$ fixed.

\begin{lemma}\label{semicont}
Fix $M$ and $M'$ given as above. For any $H \in \mathcal F_{k_0}$, the dimension $\dim_{\R}(\mathfrak{hol}_0(H))$ of the space of infinitesimal deformations of $H$ is finite. Moreover, the function $\dim_{\R}(\mathfrak{hol}_0(\cdot)): \mathcal F_{k_0} \to \N_0$ is upper semicontinuous, i.e.\ for any $H \in \mathcal F_{k_0}$ there exists a neighborhood $\mathcal V$ of $H$ in $\mathcal F_{k_0}$ such that for any $H'\in \mathcal V$ we have $\dim_{\R}(\mathfrak{hol}_0(H'))\leq \dim_{\R}(\mathfrak{hol}_0(H))$.
\end{lemma}

The following lemma follows from Theorem \ref{jetparam}, Lemma \ref{contain} and Lemma \ref{semicont} with the same proof as Lemma 23 in \cite{dSLR15}.

\begin{lemma} \label{dim10}
 Let $\Lambda_0\in A_j$,  and suppose that
  $\dim_{\mathbb R} \mathfrak {hol}_0 (\Phi_j(\Lambda_0)) = \ell$.
%i.e. $ \mathfrak {hol}_0 (\Phi(\Lambda_0)) = \mathfrak {hol}_0 (M')|_{\Phi(\Lambda_0)(\mathbb C^2)}$.
 % Let $\Lambda_0\in A$, set $\ell=\dim_{\mathbb R}\mathfrak{hol}_0(M')$ and suppose that
 %  $\dim_{\mathbb R} \mathfrak {hol}_0 (\Phi(\Lambda_0)) = \ell$, 
 %  i.e. $ \mathfrak {hol}_0 (\Phi(\Lambda_0)) = \mathfrak {hol}_0 (M')|_{\Phi(\Lambda_0)(\mathbb C^2)}$.
% \  $\mathfrak {hol}_0 (\Phi(\Lambda_0))$ 
 %only consists of the vector fields described in Remark \ref{trivsol}. 
 Then there exists a neighborhood $U$ of $\Lambda_0$ in $J_0^{\mathbf t k_0}$ such that, 
 if $X\subset A_j$ is a submanifold such that $X\cap U\neq \emptyset$, then $\dim_{\mathbb R}(X) \leq \ell$.
\end{lemma}

\section{Properties of the group action}
\label{s:propiso}

In this section we deduce some results which will be used to prove Theorem \ref{suffcon2intro}. Thus we consider strictly pseudoconvex hypersurfaces $M \subset \C^N$ and $M' \subset \C^{N'}$. In the coordinates introduced in section \ref{prelim} this means that the CR-dimension of $M$ and $M'$ are equal to $n=N-1$ and $n'=N'-1$ respectively (and $d=d'=1$).

More specifically we are interested in describing some properties of the action of the isotropy group on $2$-nondegenerate embeddings, or more precisely on the set of their $4$-jets (which by Theorem \ref{jetparam} parametrize $\mathcal F_2$). To this end we first give a brief description of the isotropy groups of the spheres $\mathbb H^{n+1}=\{(z,w) \in \C^{n+1}: {\rm Im}\, w =\|z\|^2\}$. Let $\Gamma_n = \mathbb R^+ \times \mathbb R  \times U(n) \times \mathbb C^n$ be a parameter space. Then the map
\begin{equation}\label{param}
\Gamma_n \ni \gamma =  (\lambda, r, U, c) \to \sigma_{\gamma}(z,w) = \frac{(\lambda U \ {}^t(z + c w), \lambda^2 w) }{1-2i \langle \overline c, z \rangle + (r - i \|c\|^2)w} \in {\rm Aut_0}(\mathbb H^{n+1})
\end{equation}
is a diffeomorphism between $\Gamma_n$ and ${\rm Aut_0}(\mathbb H^{n+1})$: here we denote by $\langle \cdot, \cdot \rangle$ the product on $\mathbb C^n$ given by $\langle z, \widetilde z \rangle=  \sum_{j=1}^n z_j\widetilde z_j$ and we write $\| z \|^2 = \langle \overline z, z \rangle$.

The first property we are going to study is properness: we remind the reader that the action of a topological group $\mathcal G$ on a space $X$ is said to be \emph{proper} if the map $\mathcal G\times X \to X \times X$ given by $(g,x) \mapsto (x,gx)$ is proper.

We will actually prove properness of the action on a particular subset of $J_0^4$: let $E$ be the subset of $J_0^4$ defined by  
\[E=\left\{\Lambda_{N'}^{\alpha,0} = 0, |\alpha|\leq 2,\  0 \neq \Lambda_{N'}^{0,1} = \|\Lambda'^{\beta,0}\|^2, |\beta|=1,\left\langle \overline\Lambda'^{\gamma,0}, \Lambda'^{\delta,0} \right\rangle = 0, \gamma \neq \delta, |\gamma|=|\delta|=1 \right \}.\]

One can verify in a straightforward manner the following properties of $E$:
\begin{itemize}
\item $E$ is a (real algebraic) submanifold of $J_0^4$.
\item For $M=\{{\rm Im}\, w = \|z\|^2 +O(2)\}$ and $M'=\{{\rm Im}\, w' = \|z'\|^2 + O(2)\}$ the set $E$ contains the $4$-th jet of any non-constant map from $M$ to $M'$.
\item $E$ is invariant under the action of ${\rm Aut}_0(\mathbb H^N) \times {\rm Aut}_0(\mathbb H^{N'})$, cf. Lemma 14 in \cite{dSLR15}.
\end{itemize}

\begin{lemma}\label{proact}
Suppose that $M\not \cong \mathbb H^{N}$ is strictly pseudoconvex and $M'=\mathbb H^{N'}$. Then the action of $G =  {\rm Aut}_0 (M) \times{\rm Aut}_0 (\mathbb H^{N'})$ on $E$ is proper.
\end{lemma}
\begin{proof}
The proof follows closely the one of \cite[Lemma 15]{dSLR15}. With the same argument as there using the compactness of ${\rm Aut}_0(M)$, which follows from the assumption that $M\not\cong \mathbb H^N$ (see \cite{BV}), we can reduce to showing the following: let $C>1$, and $\{(\Lambda_m, \widetilde \Lambda_m)\}_{m\in\N} \subset E  \times E$, $(\sigma'_m)_{m\in \N} \subset {\rm Aut}_0(\mathbb H^{N'})$ be sequences such that $|\Lambda_m|, |\widetilde \Lambda_m|\leq C$, $|(\Lambda_m)_{N'}^{0,1}|, |(\widetilde \Lambda_m)_{N'}^{0,1}|\geq 1/C$  and \begin{equation}\label{e:compose}\widetilde \Lambda_m = \sigma'_m \circ \Lambda_m \end{equation} for all $m\in \N$. Then $\sigma'_m$ admits a convergent subsequence. 

Using the parametrization (\ref{param}), this amounts to showing that the preimage $\{\gamma'_m=(\lambda'_m, r'_m, U'_m, c'_m)\}$ of $\sigma'_m$ in the parameter space $\Gamma'$ is relatively compact, that is $|r'_m|, \|c'_m\|, \lambda'_m$ and $1/\lambda'_m$ are bounded, since $U(N'-1)$ is a compact group.

Looking at the $N'$-th component of the first jet of \eqref{e:compose}, we get
\[{\lambda'_m}^2(\Lambda_m)_{N'}^{0,1} = (\widetilde \Lambda_m)_{N'}^{0,1},\]
and hence $\lambda'_m$ and $1/\lambda'_m$ are bounded.

% Using the parametrization (\ref{param}), this amounts to showing that the preimage $\{\gamma'_m=(\lambda'_m, r'_m, U'_m, c'_m)\}$ of $\sigma'_m$ in the parameter space $\Gamma'$ is relatively compact, that is $|r'_m|, \|c'_m\|, \lambda'_m$ and $1/\lambda'_m$ are bounded, since $U(N'-1)$ is a compact group.

Considering the first $N'-1$ components of the first jet of \eqref{e:compose} we obtain
\[\lambda_m' U_m' \ \Lambda_m'^{0,1} +(\Lambda_m)_{N'}^{0,1}\ c_m' = \widetilde \Lambda_m'^{0,1},\]
hence
\begin{align*}
\|c_m'\| \leq \frac{1}{|(\Lambda_m)_{N'}^{0,1}|}  \left\| \widetilde \Lambda_m'^{0,1} - \lambda_m' U_m' \ \Lambda_m'^{0,1}Ê\right\|, 
\end{align*}
which implies that $c_m'$ is bounded in $\mathbb C^{N'-1}$.  Finally, we consider the $N'$-th component of the second jet of \eqref{e:compose}, which gives us
\[(\widetilde \Lambda_m)_{N'}^{0,2} = -2 {\lambda'}_m^2 \left((\Lambda_m)_{N'}^{0,1}\right)^2 r'_m + R_m,\]
where $R_m$ is a polynomial expression in the second jet of $\Lambda_m$, in $\lambda'_m$ and in the coefficients of $c'_m$ and $U'_m$ (but which does not depend on $r'_m$). This shows that the sequence $r'_m$ is bounded in $\mathbb R$, and concludes the proof.\end{proof}

Next, we consider the case $M=\mathbb H^N$ and $M' \not\cong \mathbb H^{N'}$. The proof is quite similar to the previous one, but does not really reduce to it.

\begin{lemma}\label{proact2}
The action of ${\rm Aut}_0 (\mathbb H^N) \times {\rm Aut}_0 (M')$ on $E$ is proper.
\end{lemma}
\begin{proof} 
By the compactness of ${\rm Aut}_0 (M')$ as in the previous lemma it is enough to show the following: let $C>1$, and let $\{(\Lambda_m,\widetilde\Lambda_m)\}_{m\in\N} \subset E \times E$, $(\sigma_m)_{m\in \N} \subset {\rm Aut}_0 (\mathbb H^N)$ be sequences such that $|\Lambda_m|, |\widetilde \Lambda_m|\leq C$, $|(\Lambda_m)_{N'}^{0,1}|, |(\widetilde \Lambda_m)_{N'}^{0,1}|\geq 1/C$ and \begin{equation}\label{e:compose2}\widetilde\Lambda_m = \Lambda_m \circ \sigma_m\end{equation} for all $m\in \N$. Then $\sigma_m$ admits a convergent subsequence. 

The $N'$-th component of the first jet of \eqref{e:compose2} gives
\[\lambda_m^2(\Lambda_m)_{N'}^{0,1} = (\widetilde \Lambda_m)_{N'}^{0,1},\]
which implies that the sequence $\lambda_m$ is bounded above and below.

Given $\Lambda \in J_0^4$ we denote by $\Lambda'^{1,0}$ the $(N'-1)\times (N-1)$-matrix given by $(\Lambda_j'^{\alpha,0}), j=1,\ldots,N'-1, \alpha \in \N_0^{N-1}$ with $|\alpha| = 1$. The first $N'-1$ components of (the $(0,1)$-part of) the first jet of \eqref{e:compose2} can then be written as follows:
\[\lambda_m  \left (\lambda_m \Lambda_m'^{0,1} +   \Lambda_m'^{1,0} U_m c_m \right)= \widetilde \Lambda_m'^{0,1}  , \]
therefore
\[ \frac{1}{\lambda_m } \widetilde \Lambda_m'^{0,1} - \lambda_m \Lambda_m'^{0,1}=  \Lambda_m'^{1,0} U_m c_m. \]
By definition of $E$ we can write the matrix $\Lambda_m'^{1,0}$ as $\sqrt{(\Lambda_m)_{N'}^{0,1}} A_m$, where $A_m$ is a semi-unitary matrix, i.e. ${}^t \overline A_m A_m = I_{N-1}$, thus we have 
\begin{align*}
\| \Lambda_m'^{1,0} U_m c_m \|^2 = (\Lambda_m)_{N'}^{0,1} \| A_m U_m c_m\|^2 = (\Lambda_m)_{N'}^{0,1} \| U_m c_m \|^2 = (\Lambda_m)_{N'}^{0,1} \|c_m \|^2,
\end{align*}
so that by the estimate on $(\Lambda_m)_{N'}^{0,1}$ it holds that
\[ \frac{\|c_m\| }{\sqrt C}\leq  \lambda_m \left \|\ \Lambda_m'^{0,1} \right \| +\frac{1}{\lambda_m } \left \| \widetilde \Lambda_m'^{0,1}\right \|,\]
which implies the boundedness of $c_m$ in $\C^{N-1}$. Finally, we consider the $N'$-th component of the second jet of \eqref{e:compose2}, which gives the equation
\[(\widetilde \Lambda_m)_{N'}^{0,2} = -\lambda_m^2 (\Lambda_m)_{N'}^{0,1} r_m + R_m,\]
where $R_m$ is a polynomial expression in the second jet of $\Lambda_m$ and in $\lambda_m$, $c_m$, $U_m$, not depending on $r_m$. This implies that the sequence $r_m$ is bounded in $\mathbb R$, and concludes the proof.
\end{proof}

Next we are going to prove the freeness of the action of the isotropy group of the target manifold. In order to do so, we first introduce for any fixed map $H$, in a way similar to Lemma 17 in \cite{dSLR15}, coordinates such that 
\begin{itemize}
\item the map $H$ is of the form $(z,F(z,w),w)$ for a certain germ of holomorphic function $F: \mathbb C^{N}\to \mathbb C^{\ell}$, where $\ell = N'-N$, such that $F(0) = 0$;
\item the automorphism group of $M'$ at $0$ is a subgroup of ${\rm Aut}_0 (\mathbb H^{N'})$. 
\end{itemize}

\begin{lemma}\label{free}
Let $\Lambda\in E$ be the $4$-jet of a map of the form $(z,w)\to(z,F(z,w),w) \in \mathcal F_2$. Then the stabilizer of $\Lambda$ under the action of $G' = \{\id\} \times {\rm Aut}_0(\mathbb H^{N'})$ is trivial.
\end{lemma}
\begin{proof}
For $v \in \C^m$ we denote by $v_{[j;k]}$ the coordinates $(v_j, \ldots, v_k)$ for $1\leq j \leq k\leq m$. First using that $\Lambda\in E$ we deduce $\Lambda_N^{\alpha,0} = \cdots =\Lambda_{N'-1}^{\alpha,0} =0$ for all $\alpha$ such that $|\alpha|=1$. Indeed for all $\alpha$ with $|\alpha|=1$ we have
\[
1 = \Lambda_{N'}^{0,1} = \|\Lambda'^{\alpha,0}\|^2 = \sum_{j=1}^{N-1} |\Lambda_j^{\alpha,0}|^2 + \sum_{j=N}^{N'-1} |\Lambda_j^{\alpha,0}|^2  = 1 + \sum_{j=N}^{N'-1} |\Lambda_j^{\alpha,0}|^2,
\]
since for every $\alpha$ there exists exactly one $j_{\alpha}$, such that $\Lambda_{j_\alpha}^{\alpha,0} = 1$ and $\Lambda_{j}^{\alpha,0} = 0$ for all $j \neq j_{\alpha}, 1\leq j \leq N-1$ (by the form of $\Lambda$). In other words the entries of the last $\ell$ rows of the $(N-1+\ell)\times (N-1)$ matrix $\Lambda'^{1,0}$ defined in the proof of Lemma \ref{proact2} are all zeros, while the rest of the matrix is an $(N-1)\times (N-1)$ identity matrix.
 
  Let $\Lambda$ be as in the assumptions and $\sigma' \in G'$ in the stabilizer of $\Lambda$, that is
\begin{equation}
\label{eq:free}
\Lambda = \sigma' \circ \Lambda.
\end{equation}
 
Let $\sigma'$ be parametrized by $\gamma' = (\lambda',r',U',c')\in \Gamma'$. We will show that $\sigma' = \id$ by following similar computations analogous to the ones in Lemma \ref{proact}. Looking at the $N'$-th component of the first jet of \eqref{eq:free} we see that $(\lambda')^2 = 1$, hence $\lambda' = 1$ since $\lambda'\in \mathbb R^+$.

Considering the first $N'-1$ components of the first jet of \eqref{eq:free}, we get $U'  \Lambda'^{1,0} =\Lambda'^{1,0}$, hence $U'$ is a block diagonal matrix with first block being $I_{N-1}$  and the second block a $\ell \times \ell$ unitary matrix $U''$. Furthermore, we obtain
\begin{align*}
U' \ {}^t\left( 0, \ldots, 0, \Lambda_{[N;N'-1]}^{0,1} \right) +  c'  ={}^t\left( 0, \ldots, 0, \Lambda_{[N;N'-1]}^{0,1} \right),
\end{align*}
i.e. $c_j' = 0$ for $1\leq j \leq N-1$ and $c'_{[N;N'-1]} = (I_{\ell}-U'')\Lambda_{[N;N'-1]}^{0,1}$.

Since $\Lambda$ comes from a map in $\mathcal F_2$ it follows that there exists a collection of multiindices $\Delta=(\delta_1,\ldots, \delta_{\ell})$ with $|\delta_j|=2$ such that the $\ell\times \ell$-matrix $\Lambda_{[N;N'-1]}^{\Delta,0}$ (whose $(j,k)$-entry is $\Lambda_{j+N-1}^{\delta_k,0}$) is invertible. Indeed in the Definition \ref{defNondeg} we can take the sequence of multiindices $(\iota_1,\ldots, \iota_{N'})$ to be $\iota_1 = 0$, for $2 \leq k \leq N$ the multiindex $\iota_k = (0,\ldots, 0,1,0,\ldots, 0)$ (where the $1$ appears in the $k-1$-th entry) and $\iota_{N+j}=\delta_j$ for $1\leq j \leq \ell$.
With this choice the determinant $s$ at $0$ of the matrix in Definition \ref{defNondeg} is
 \[s(0) = \det \left(\begin{array}{cccccc} 
0 & 0 & \cdots & 0 & 0 & 1 \\ 
1 & 0 & \cdots & 0  &0 & 0 \\ 
0 & 1 & \cdots &  0 & 0 & 0 \\
\vdots & \vdots &\ddots & \vdots & \vdots & \vdots  \\
0 & 0 & \cdots & 1 & 0 & 0 \\
\Lambda_1^{\Delta,0} &  \Lambda_{2}^{\Delta,0} & \cdots &  \Lambda_{N-1}^{\Delta,0} & \Lambda_{[N;N'-1]}^{\Delta,0}  & 0
 \end{array}\right)
 = \pm \det \Lambda_{[N;N'-1]}^{\Delta,0},\]
 so that $ \det \Lambda_{[N;N'-1]}^{\Delta,0} \neq 0$. Using this fact and considering the $[N;N'-1]$ components of the second jet of \eqref{eq:free} we have:
\begin{align*}
U''\Lambda_{[N;N'-1]}^{\Delta} & =\Lambda_{[N;N'-1]}^{\Delta,0}, \  {i.e.}\  U'' = I_{\ell}, c_{[N;N'-1]}' = 0,
\end{align*}
since $\Lambda_{[N;N'-1]}^{\Delta}$ is invertible. Finally taking into account the previous computations the remaining equation in the $2$-jet of the $N'$-th component in \eqref{eq:free} becomes
\begin{align*}
-r' & = \Lambda_{N'}^{0,2} = 0,
\end{align*}
which shows that $\sigma'=\id$. 
\end{proof}

\section{Proofs of the main results}
\label{proof}

In this section we are going to prove Theorems \ref{infTrivial} and \ref{suffcon2intro}.  

\begin{theorem}\label{suffcon1}
Let $M,M'$ be as in Theorem \ref{jetparam}, let $A$ be defined as in (\ref{defas}), $\Lambda_0\in A$, and let $j\in J$ be such that $\Lambda_0 \in A_j$. Suppose that $\dim_{\mathbb R} \mathfrak {hol}_0 (\Phi_j(\Lambda_0)) = 0$, then $\Phi_j(\Lambda_0)$ is isolated in $\mathcal F_{k_0}$.
\end{theorem}

\begin{proof}
% Note that, by Remark \ref{trivsol}, the assumption implies in particular that $\dim_{\mathbb R}\mathfrak{hol}_0(M')=0$. 
By Lemma \ref{dim10}, there is a neighborhood $U$ of $\Lambda_0$ in $J_0^{\mathbf t k_0}$ such that $U\cap A_j$ does not contain any manifold of positive 
dimension. It follows that $U\cap A_j$ is a discrete set: by Lemma \ref{PhiHomo} we have that $\Phi_j(\Lambda_0)$ is in turn isolated in $\mathcal F_{k_0}$.
\end{proof}

Now we have all the ingredients to give a proof of Theorem \ref{infTrivial}:

\begin{proof}[Proof of Theorem \ref{infTrivial}]
Let $H$ be a finitely nondegenerate map. Then there exists an integer $k_0$ such that $H \in \mathcal F_{k_0}$ (see Definition \ref{defNondeg}). Set $\Lambda_0 = j_0^{\mathbf t k_0} H$, then there exists $j\in J$, such that $\Lambda_0 \in A_j$ and $H = \Phi_j(\Lambda_0)$. By Theorem \ref{suffcon1} $H$ is isolated in $\mathcal F_{k_0}$, but since $\mathcal F_{k_0}$ is an open set of $\mathcal H(M,M')$ (see again Definition \ref{defNondeg}), it follows that $H$ is isolated in $\mathcal H(M,M')$. In particular $H$ is locally rigid by Remark \ref{rem:equcon}.
\end{proof}

Let us now turn to Theorem \ref{suffcon2intro}.
In the setting the jet space in Theorem \ref{jetparam} can be taken to be $J_0^4$, since $k_0 = 2$ and $\mathbf t = 2$. The proof of Theorem \ref{suffcon2intro} follows from the theorem below in a similar way as Theorem \ref{infTrivial} from Theorem \ref{suffcon1}.

\begin{theorem}\label{suffcon2}
Let $M,M'$ be as in Theorem \ref{suffcon2intro} and let $A \subset J_0^4$ be defined as in (\ref{defas}), and let $\Lambda_0\in A$. Let $j \in J$ be such that $\Lambda_0 \in A_j$ and $\dim_{\mathbb R} \mathfrak {hol}_0 (\Phi_j(\Lambda_0)) = \dim_{\mathbb R}\mathfrak{hol}_0(M') = \ell$. Then $\Phi_j(\Lambda_0)$ is locally rigid.
\end{theorem}

\begin{proof}
The result can be proved using the properness and freeness of the action of isotropies according to the Lemmas \ref{proact}, \ref{proact2} and \ref{free}. This allows to employ the local slice theorem for free and proper actions and the conclusion can be obtained by arguing exactly as in the proof of Theorem 25 in \cite{dSLR15}.
\end{proof}


\begin{bibdiv}
\begin{biblist}

\bib{BER1}{article}{
   author={Baouendi, M. S.},
   author={Ebenfelt, P.},
   author={Rothschild, L. P.},
   title={Algebraicity of holomorphic mappings between real algebraic sets
   in ${\bf C}^n$},
   journal={Acta Math.},
   volume={177},
   date={1996},
   number={2},
   pages={225--273},
   issn={0001-5962},
   review={\MR{1440933 (99b:32030)}},
   doi={10.1007/BF02392622},
}

\bib{BER2}{book}{
   author={Baouendi, M. S.},
   author={Ebenfelt, P.},
   author={Rothschild, L. P.},
   title={Real submanifolds in complex space and their mappings},
   series={Princeton Mathematical Series},
   volume={47},
   publisher={Princeton University Press},
   place={Princeton, NJ},
   date={1999},
   pages={xii+404},
   isbn={0-691-00498-6},
   review={\MR{1668103 (2000b:32066)}},
}

\bib{BER99}{article}{
   author={Baouendi, M. S.},
   author={Ebenfelt, P.},
   author={Rothschild, L. P.},
   title={Rational dependence of smooth and analytic CR mappings on their
   jets},
   journal={Math. Ann.},
   volume={315},
   date={1999},
   number={2},
   pages={205--249},
   issn={0025-5831},
   review={\MR{1721797 (2001b:32075)}},
   doi={10.1007/s002080050365},
}

\bib{BV}{article}{
   author={Beloshapka, V. K.},
   author={Vitushkin, A. G.},
   title={Estimates of the radius of convergence of power series that give
   mappings of analytic hypersurfaces},
   language={Russian},
   journal={Izv. Akad. Nauk SSSR Ser. Mat.},
   volume={45},
   date={1981},
   number={5},
   pages={962--984, 1198},
   issn={0373-2436},
   review={\MR{637612 (83f:32017)}},
}


\bib{CH}{article}{
   author={Cho, Chung-Ki},
   author={Han, Chong-Kyu},
   title={Finiteness of infinitesimal deformations of CR mappings of CR
   manifolds of nondegenerate Levi form},
   journal={J. Korean Math. Soc.},
   volume={39},
   date={2002},
   number={1},
   pages={91--102},
   issn={0304-9914},
   review={\MR{1872584 (2002j:32036)}},
   doi={10.4134/JKMS.2002.39.1.091},
}

\bib{Da}{article}{
   author={D'Angelo, John P.},
   title={Proper holomorphic maps between balls of different dimensions},
   journal={Michigan Math. J.},
   volume={35},
   date={1988},
   number={1},
   pages={83--90},
   issn={0026-2285},
   review={\MR{931941 (89g:32038)}},
   doi={10.1307/mmj/1029003683},
}

\bib{dSLR15}{article}{
  author = {{Della Sala}, Giuseppe and Lamel, Bernhard and Reiter, Michael},
	doi = {10.1090/tran/6885},
	fjournal = {Transactions of the American Mathematical Society},
	issn = {0002-9947},
	journal = {Trans. Amer. Math. Soc.},
	mrclass = {32H02 (32V40)},
	mrnumber = {3695846},
	number = {11},
	pages = {7829?7860},
	title = {Local and infinitesimal rigidity of hypersurface embeddings},
	url = {http://dx.doi.org/10.1090/tran/6885},
	volume = {369},
	year = {2017},
}

\bib{Fa2}{article}{
   author={Faran, James J.},
   title={Maps from the two-ball to the three-ball},
   journal={Invent. Math.},
   volume={68},
   date={1982},
   number={3},
   pages={441--475},
   issn={0020-9910},
   review={\MR{669425 (83k:32038)}},
   doi={10.1007/BF01389412},
}

\bib{Hu}{article}{
   author={Huang, Xiaojun},
   title={On a linearity problem for proper holomorphic maps between balls
   in complex spaces of different dimensions},
   journal={J. Differential Geom.},
   volume={51},
   date={1999},
   number={1},
   pages={13--33},
   issn={0022-040X},
   review={\MR{1703603 (2000e:32020)}},
}

\bib{HJ}{article}{
   author={Huang, Xiaojun},
   author={Ji, Shanyu},
   title={Mapping $\bold B^n$ into $\bold B^{2n-1}$},
   journal={Invent. Math.},
   volume={145},
   date={2001},
   number={2},
   pages={219--250},
   issn={0020-9910},
   review={\MR{1872546 (2002i:32013)}},
   doi={10.1007/s002220100140},
}

\bib{Ji}{article}{
   author={Ji, Shanyu},
   title={A new proof for Faran's theorem on maps between $\Bbb B^2$ and
   $\Bbb B^3$},
   conference={
      title={Recent advances in geometric analysis},
   },
   book={
      series={Adv. Lect. Math. (ALM)},
      volume={11},
      publisher={Int. Press, Somerville, MA},
   },
   date={2010},
   pages={101--127},
   review={\MR{2648940 (2011c:32023)}},
}
		




\bib{JL}{article}{
   author={Juhlin, Robert},
   author={Lamel, Bernhard},
   title={Automorphism groups of minimal real-analytic CR manifolds},
   journal={J. Eur. Math. Soc. (JEMS)},
   volume={15},
   date={2013},
   number={2},
   pages={509--537},
   issn={1435-9855},
   review={\MR{3017044}},
   doi={10.4171/JEMS/366},
}


\bib{La}{article}{
   author={Lamel, Bernhard},
   title={Holomorphic maps of real submanifolds in complex spaces of
   different dimensions},
   journal={Pacific J. Math.},
   volume={201},
   date={2001},
   number={2},
   pages={357--387},
   issn={0030-8730},
   review={\MR{1875899 (2003e:32066)}},
   doi={10.2140/pjm.2001.201.357},
}

\bib{Le}{article}{
   author={Lebl, Ji{\v{r}}{\'{\i}}},
   title={Normal forms, Hermitian operators, and CR maps of spheres and
   hyperquadrics},
   journal={Michigan Math. J.},
   volume={60},
   date={2011},
   number={3},
   pages={603--628},
   issn={0026-2285},
   review={\MR{2861091}},
   doi={10.1307/mmj/1320763051},
}


 \bib{Re2}{article}{
   author = {Reiter, Michael},
	doi = {10.1007/s12220-015-9594-6},
	fjournal = {Journal of Geometric Analysis},
	issn = {1050-6926},
	journal = {J. Geom. Anal.},
	mrclass = {32H02 (32V30)},
	mrnumber = {3472839},
	number = {2},
	pages = {1370-1414},
	title = {Classification of holomorphic mappings of hyperquadrics from {$\Bbb{C}^2$} to {$\Bbb{C}^3$}},
	url = {http://dx.doi.org/10.1007/s12220-015-9594-6},
	volume = {26},
	year = {2016},
   }  
   
\bib{Re3}{article}{
   author = {Reiter, Michael},
	doi = {10.2140/pjm.2016.280.455},
	fjournal = {Pacific Journal of Mathematics},
	issn = {0030-8730},
	journal = {Pacific J. Math.},
	mrclass = {32H02 (32V30 57S05 57S25 58D19)},
	mrnumber = {3453979},
	number = {2},
	pages = {455-474},
	title = {Topological aspects of holomorphic mappings of hyperquadrics from {$\mathbb{C}^2$} to {$\mathbb{C}^3$}},
	url = {http://dx.doi.org/10.2140/pjm.2016.280.455},
	volume = {280},
	year = {2016},
   }
  

\bib{We}{article}{
   author={Webster, S. M.},
   title={The rigidity of C-R hypersurfaces in a sphere},
   journal={Indiana Univ. Math. J.},
   volume={28},
   date={1979},
   number={3},
   pages={405--416},
   issn={0022-2518},
   review={\MR{529673 (80d:32022)}},
   doi={10.1512/iumj.1979.28.28027},
}




\end{biblist}
\end{bibdiv}
\end{document}